\documentclass{article}

\usepackage{amsmath}
\usepackage{amsfonts}
\usepackage{amsthm}
\usepackage{amssymb}
\usepackage{bm}
\usepackage[applemac]{inputenc}

\usepackage{comment}

\usepackage[T1]{fontenc}
\usepackage{todonotes}
\usepackage{enumerate}
\makeatletter

\newtheoremstyle{mystyle}{}{}{\slshape}{2pt}{\scshape}{.}{ }{}

\newtheorem{thm}{Theorem}[section]
\newtheorem{corollary}[thm]{Corollary}

\newtheorem{theorem}[thm]{Theorem}

\newtheorem{claim}[thm]{Claim}

\newtheorem{cor}[thm]{Corollary}

\newtheorem{prop}[thm]{Proposition}
\newtheorem{lemme}[thm]{Lemma}
\newtheorem{lemma}[thm]{Lemma}

\newtheorem{fact}[thm]{Fact}

\newtheorem{obs}[thm]{Observation}
\theoremstyle{definition}
\newtheorem{defi}[thm]{Definition}
\newtheorem{definition}[thm]{Definition}
\theoremstyle{mystyle}

\theoremstyle{remark}
\newtheorem{rem}[thm]{Remark}

\newenvironment{claimproof}[1]{\par\noindent\emph{Proof:}\space#1}{\hfill $\blacksquare_{Claim}$ \bigskip}

\DeclareMathOperator{\uth}{U^\text{\th}}

\newcommand{\ignore}[1]{}

\DeclareMathOperator{\tp}{tp}
\DeclareMathOperator{\dlr}{op-dim}

\DeclareMathOperator{\dcl}{dcl}
\DeclareMathOperator{\acl}{acl}

\DeclareMathOperator{\rk}{rk}

\DeclareMathOperator{\thr}{\text{\th}}

\DeclareMathOperator{\qfL}{qftp^\mathcal L}
\DeclareMathOperator{\qfr}{qftp^{\{R\}}}

\def\indsym#1#2{%
 \setbox0=\hbox{$\m@th#1x$}%
 \kern\wd0%
 \hbox to 0pt{\hss$\m@th#1\mid$\hbox to 0pt{$\m@th#1^{#2}$\hss}\hss}%
 \lower.9\ht0\hbox to 0pt{\hss$\m@th#1\smile$\hss}%
 \kern\wd0}
\newcommand{\ind}[1][]{\mathop{\mathpalette\indsym{#1}}}

\def\nindsym#1#2{%
 \setbox0=\hbox{$\m@th#1x$}%
 \kern\wd0%
 \hbox to 0pt{\hss$\m@th#1\not$\kern1.4\wd0\hss}
 \hbox to 0pt{\hss$\m@th#1\mid$\hbox to 0pt{$\m@th#1^{#2}$\hss}\hss}%
 \lower.9\ht0\hbox to 0pt{\hss$\m@th#1\smile$\hss}%
 \kern\wd0}
\newcommand{\nind}[1][]{\mathop{\mathpalette\nindsym{#1}}}

\DeclareMathOperator{\thind}{\ind^{\thr}}
\DeclareMathOperator{\nthind}{\nind^{\thr}}

\title{Dependent finitely homogeneous rosy structures}

\author{Alf Onshuus, Pierre Simon\footnote{Partially supported by NSF (grants no. 1665491 and 1848562).}}

\date{}

\theoremstyle{plain}
\newtheorem*{thm*}{Theorem}

\begin{document}

\maketitle

\begin{abstract}
We study finitely homogeneous dependent rosy structures, adapting results of Cherlin, Harrington and Lachlan proved for $\omega$-stable $\omega$-categorical structures. In particular, we prove that such structures have finite \th-rank and are coordinatized by a \th-rank 1 set. We show that they admit a distal, finitely axiomatizable, expansion.

These results show that there are, up to inter-definability, at most countably many
dependent rosy structures $M$ which are homogeneous in a finite relational language.
\end{abstract}

\section{Introduction}
The work towards classifying totally categorical structures is considered by many to be the starting point of ``geometric stability theory''.
This began with the characterization of $\omega$-categorical strictly minimal sets (proved by Zilber using combinatorial and model-theoretic methods, and by Cherlin -- Mills using group theory) as either degenerate geometries, or affine or projective geometries over finite fields. Cherlin, Harrington and Lachlan then proved in \cite{CHL} the Coordinatization Theorem for $\omega$-categorical $\omega$-stable theories which made it possible to analyze $\omega$-categorical $\omega$-stable structures in terms of strictly minimal sets. This was then used by Ahlbrandt and Ziegler (\cite{AZ1}) followed by Hrushovski who proves in \cite{Hr} that any totally categorical theory is quasi-finitely axiomatizable and classifies such theories in the disintegrated case.

In \cite{Si}, a characterization of unstable \th-rank 1 sets in $\omega$-categorical dependent theories was given, providing a starting point analogue to the one given by Zilber and  Cherlin -- Mills in the stable case. This opens a path towards generalizing the results of \cite{CHL} and \cite{Hr} to dependent theories. The analogue of the rank given by $\omega$-stability is the thorn-rank (denote \th-rank) and introduced in \cite{On}. Theories for which this rank is ordinal-valued are called superrosy. Hence we expect that it is possible to generalize the aforementioned results to the class of superrosy dependent $\omega$-categorical theories. In particular, all such theories should be quasi-finitely axiomatizable. In this paper, we make a first step towards this. We avoid difficulties coming from proving finiteness of the rank and non-trivial strongly minimal sets by replacing the $\omega$-categorical hypothesis by the stronger one of being finitely homogeneous, namely having quantifier elimination in a finite relational language. For such theories, we show that being merely rosy implies having finite \th-rank.

Recall from \cite{CHL} that a $\emptyset$-definable set $X$ coordinatizes $M$ if for every $a\in M$, $\acl(a)\cap X \neq \emptyset$. If $M$ is primitive, it follows that $M$ is isomorphic to the set of conjugates in $X$ of some finite subset of $X$ (a structure called a \emph{grassmanian} of $X$.)

The central result from this paper is Theorem \ref{th:coord}, which generalizes \cite[Theorem 4.1]{CHL} to our context.

\begin{thm*}[Coordinatization]
Let $M$ be finitely homogeneous, dependent and rosy. Then $M$ has finite rank and is coordinatized by a \th-rank 1 formula.
\end{thm*}



The second main theorem proved in Section \ref{sec:distal exp} is an analogue of Lachlan's result in \cite{Lac84} that a finitely homogeneous stable structure is interpretable in dense linear orders.

\begin{thm*}[Distal expansion]
Let $M$ be a finitely homogeneous dependent rosy structure, then $M$ admits an expansion which is distal and finitely axiomatizable. In particular, there are only countably many such structures up to interdefinability.
\end{thm*}

Finally, note that by \cite[Lemma 7.1]{Si} any structure which eliminates quantifiers in a binary language has finite \th-rank. It follows that our results apply in particular to dependent structures homogeneous in a finite binary language.

\section{Preliminaries}

We use standard model theoretic notation. We work a structure $M$ in a language $L$. Variables such as $x,y,z,\ldots$ and $a,b,c,\ldots$ usually denote finite tuples, although this is sometimes made explicit by writing say $\bar a,\bar b,\ldots$.

We usually work in $M^{eq}$, so that $\dcl$ and $\acl$ mean $\dcl^{eq}$ and $\acl^{eq}$ respectively. We might nonetheless use the latter notation to emphasize that imaginaries are allowed.

Recall that a theory $T$ is $\omega$-categorical if any two countable models of $T$ are isomorphic.
A countable structure $M$ in a relational language $L$ is \emph{homogeneous} if
for any finite $A\subseteq M$ and any partial $L$-isomorphism $\sigma\colon A\rightarrow M$, there is an automorphism $\bar \sigma\colon M\rightarrow M$ that extends $\sigma$.
A structure is \emph{finitely homogeneous} if it is homogeneous in a finite relational language. Any such structure is $\omega$-categorical.

\begin{defi}
A formula $\phi(x;y)$ witnesses the independence property in $M$ if we can find tuples $( a_i)_{i\in \omega}$ and $( b_J)_{J\in \mathcal P(\omega)}$ such that
\[M\models \phi(a_i;b_J) \Leftrightarrow i\in J.\]

A theory $T$ is \emph{dependent} (or \emph{NIP}) if no formula witnesses the independence property in any model of $T$.
\end{defi}

\subsection{Rosiness}

We will need to introduce definitions and basic results, most of which come from \cite{On}.

\begin{defi}
A formula $\phi(x,a)$ \emph{strongly divides} over $B$ if
$\tp(a/B)$ is not algebraic and $\{\phi(x,a')\}_{a'\models
tp(a/B)}$ is $k$-inconsistent for some $k$.

A formula \emph{\th-divides} over $A$ if it strongly divides over some
$B\supseteq A$.

A formula \emph{\th-forks} over $A$ if it implies a finite
disjunction of formulas each of which \th-divides over $A$.

A type \emph{\th-forks} over $A$ if it implies a formula which \th-forks
over $A$.
\end{defi}

We will now introduce three notions of rank all based in \th-forking. Since we are working in $\omega$-categorical structures
complete types are definable sets so that many of the details which differentiate the different notions don't happen in our context.

\medskip

The first rank is \th-rank is the foundational rank defined on definable sets induced by \th-forking.

\begin{definition}
\th-rank is
the least function from the set of definable sets into the ordinals (and $\infty$)
which satisfies that
\begin{itemize}
\item any consistent formula $\theta(x)$ has \th-rank greater than 0, with equality
if and only if $\theta(x)$ is algebraic, and

\item if $\phi(x,a)$ is a formula, then $\thr(\phi(x,a))\geq \alpha+1$ if
some $\theta(x,b)$ implies $\phi(x,a)$ and \th-forks over $a$.
\end{itemize}

We say that a theory is \emph{superrosy} if the \th-rank is ordinal valued.
\end{definition}

\begin{fact}
In a superrosy theory (which will be one of our underlying assumptions), the following two results hold:
\begin{itemize}
\item If $q(x)$ does nor fork over $A$ then $\thr(q(x))=\thr(p(x))$.

\item $\thr\left(\phi\left(x,a\right)\wedge
\theta\left(x,b\right)\right)=\max\left(\thr\left(\phi\left(x,a\right)\right),\thr\left(\theta\left(x,b\right)\right)\right).
$
\end{itemize}
\end{fact}
Both were proved in \cite{On} for the local \th-ranks (which we will define next) but the proof
holds for the global \th-rank in a superrosy theory. The second property
is generally known as ``additivity''.

\medskip

The local versions of \th-rank are defined as follows:

\begin{definition}
Given a two fixed set of formulas $\Delta$ and $\Pi$ an integer $k$ and a partial type $\phi$, the \th-rank of
$\phi$ with respect to $\Delta$,$\Pi$ and $k$, denoted $\thr(\phi, \Delta, \Pi, k)$ is defined inductively as follows:
\begin{itemize}
\item $\thr(\phi, \Delta, \Pi, k)\geq 0$ if $\phi$ is consistent.
\item For $\lambda$ a limit ordinal, $\thr(\phi, \Delta, \Pi, k)\geq \lambda $ if and only if $\thr(\phi, \Delta, \Pi, k)\geq \alpha$ for any $\alpha<\lambda$.
\item $\thr(\phi, \Delta, \Pi, k)\geq \alpha+1$ if and only if there is a $\delta(x,y)\in \Delta$ and some $\pi(y,z)\in \Pi$ and a parameter $c$ such that
\begin{itemize}
\item $\thr(\phi\wedge \delta(x,a), \Delta, \Pi, k)\geq \alpha$ for infinitely many $a\models \pi(a,c)$ and
\item $\{\delta(x,a)\}_{a\models \pi(y,c)}$ is $k$-inconsistent.
\end{itemize}
\item $\thr(\phi, \Delta, \Pi, k)=\infty$ if $\thr(\phi, \Delta, \Pi, k)>\alpha$ for any ordinal $\alpha$.
\end{itemize}
\end{definition}

By compactness the local \th-ranks are either finite or non ordinal valued for finite $\Delta, \Pi, k$, and these ranks already code \th-forking (\cite{On}), in the sense that if $\thr\left(\tp\left(a/A\right), \Delta, \Pi, k\right)\neq \infty$ then $a\thind_A b$ if and only if for any finite $\Delta,\Pi$ and any $k$
\[
\thr\left(\tp\left(a/Ab\right), \Delta, \Pi, k\right)=\thr\left(\tp\left(a/A\right), \Delta, \Pi, k\right).
\]

\medskip

\begin{fact}\label{Lascar}
If $T$ is a theory of finite \th-rank, then for any $a,b$ and any fintie
set $A$ we have
\[\thr\left(ab/A\right)=\thr\left(a/bA\right)+\thr\left(b/A\right).
\]
\end{fact}

\begin{proof}
In $\omega$-categorical superrosy theories the $\uth$-rank (see \cite{On} for definition) is equal to the \th-rank for complete types over finite sets of parameters. The result
for $\uth$-rank is known as the ``Lascar inequalities'' (because it is an analogue to
the corresponding result in simple theories) and is proved in \cite{On}.
\end{proof}

\begin{obs}\label{U=Uth}
If $p(x)$ be a stable type in a superrosy $\omega$-categorical theory. Then $\thr(p)=MR(p)$ for any type
$p(x)$.
\end{obs}

An important tool in \cite{CHL} is the concept of normal formulas. In the $\omega$-stable $\omega$-categorical context
a formula $\phi(\bar x; \bar y)$ is \emph{normal} if for any $\bar b_0, \bar b_1$ we have
$\phi(\bar x, \bar b_0)\Leftrightarrow \phi(\bar x, \bar b_1)$ whenever the Morley rank of
the symmetric difference of $\phi(\bar x, \bar b_0)$ and $\phi(\bar x, \bar b_1)$ is smaller than
the Morley rank of $\phi(\bar x, \bar b_0)$.

The following is Lemma 2 in \cite{La}.

\begin{fact}\label{La}
Let $M$ be a $\aleph_0$-categorical structure, and let $\phi(\bar x, \bar y)$ be a formula which implies
a complete type in $\bar y$ and such that for some $\bar b$ $\phi(\bar x, \bar b)$ defines a set of finite Morley
rank and Morley degree 1. Then there is a normal formula $\phi^*(\bar x, \bar y)$ such that
\[MR\left(\phi\left(\bar x, \bar b\right)\right)=MR\left(\phi\left(\bar x, \bar b\right)\wedge \phi^*\left(\bar x, \bar b\right)\right)=MR\left(\phi^*\left(\bar x, \bar b\right)\right).\]
\end{fact}

\subsection{$\omega$-categorical dependent structures.}

We assume throughout the paper that we are
working in an $\omega$-categorical dependent structure $M$.

The paper \cite{Si} develops a theory of $\omega$-categorical linear orders and then uses it along with the main result of \cite{Si1} to classify primitive dependent $\omega$-categorical structures of rank 1. A certain familiarity with the results of \cite{Si} will be needed to follow some arguments in the current paper, though we recall here the basic definitions and theorems that we need.

We start with some results on $\omega$-categorical definable linear orders, that is orders $(V,\leq)$ where both $V$ and the order relation $\leq$ are definable in $M$. If $(V,\leq)$ is a definable order, the \emph{reverse} of $V$ is the definable order $(V,\geq)$.

In the following definition, an equivalence relation $E$ on a linear order $(V,\leq)$ is \emph{convex} if every equivalence class is a convex subset of $V$.

\begin{defi}
	Let $(V,\leq)$ be an $A$-definable linear order.
We say that $(V,\leq)$ is \emph{minimal} over $A$ if:
	
	\begin{itemize}
	\item the order is dense;
	\item any $A$-definable subset of $V$ is either empty or dense in $V$;
	\item $V$ does not admit any definable (over any set of parameters) convex equivalence relation $E$ with infinitely many infinite classes.
	\end{itemize}
\end{defi}

Note that the third bullet is automatically satisfied by orders of \th-rank 1.

Closures of definable subsets of minimal orders are particularly simple:

\begin{fact}[\cite{Si} Corollary 3.14]\label{fact:subsets of minimal}
	Let $(V,\leq)$ be a minimal definable linear order over some $A$. Let $X\subseteq V^n$ be an $A$-definable subset, then the topological closure of $X$ is a boolean combination of sets of the form $x_i \leq x_j$.
\end{fact}

Over larger sets of parameters, the situation is not much more complicated: we only have to take into account extra definable {cuts} in $V$. We define this now.

\begin{defi}
Let $(V,\leq)$ be an $A$-definable dense order with no first or last element. By a \emph{cut} in $V$ we mean an initial segment of it which is neither empty nor the whole of $V$ and has no last element.

We let $\overline V$ be the set of definable (over any parameters) cuts of $V$.
\end{defi}

The set $\overline V$ is naturally a union (or rather direct limit) of interpretable sets. We refer to \cite[Section 3.1]{Si} for details. In our situation, like in the rank 1 case, $\overline V$ will actually be naturally an interpretable set and the reader might prefer to think of it as such. Also $V$ is canonically included as a dense subset of $\overline V$, the point $a$ being sent to the cut $\{x\in V:x<a\}$.

\begin{definition}
	let $A\subseteq M$ and let $(V,\leq)$ be an $A$-definable linear order. For $k<\omega$, an \emph{$A$-sector} of $V^k$ is a subset of $V^k$ defined by a formula $\phi(x_0,\ldots,x_{k-1})$ which is a finite boolean combination of relations of the form:
	
	\begin{itemize}
		\item $x_i = x_j$, for $i,j<k$;
		\item $x_i \leq x_j$, for $i,j<k$;
		\item $x_i = a$, for $i<k$ and $a\in \dcl(A)\cap V$;
		\item $x_i \leq a$, for $i<k$ and $a\in \dcl(A)\cap \overline V$.
	\end{itemize}
\end{definition}

In other words, an $A$-sector is a subset of $V^k$ which is quantifier-free definable from the order along with unary predicates for $A$-definable cuts of $V$.

The following is a special case of \cite[Proposition 3.34]{Si}. We will give the general case below.

\begin{fact}\label{fact:subset one order} Let $V$ be a minimal $A$-definable orders. Let $X\subseteq V^{k}$ be an $A$-definable set. Then the closure of $X$ is an $A$-sector of $V^{k}$.
\end{fact}

We now describe the situation when more than one order is involved.

\begin{defi}
Let $(V,\leq_V)$ and $(W,\leq_W)$ be orders, definable and minimal over $A$. We say that they are \emph{intertwined}  if there is a parameter-definable increasing map $f\colon V\to \overline W$. If $A$ is clear from the context, we omit it.
\end{defi}

If $V$ and $W$ are intertwined, then the intertwining map is unique and hence is definable over $A$. Also by minimality, that map $f$ must have an image dense in $W$. It follows that $f$ induces an increasing bijection between the completions $\overline V$ and $\overline W$. Hence intertwined orders should be thought of as orders with isomorphic completions, or equivalently orders which are dense subsets of a common linear order. (See \cite[Section 3.1]{Si}.)

The opposite notion is that of {independent orders}:

\begin{defi}
Let $V$ and $W$ be two orders, definable and minimal over some $A$. We say that $V$ and $W$ are \emph{independent} if there does not exist:

$\bullet$ a set of parameters $B\supseteq A$,

$\bullet$ $B$-definable infinite subsets $X\subseteq V$ and $Y\subseteq W$, which we equip with the induced orders from $V$ and $W$ respectively,

$\bullet$ a $B$-definable intertwining from $X$ to either $Y$ or the reverse of $Y$.
\end{defi}

By \cite[Lemma 3.18]{Si}, any two disjoint convex subsets of a minimal order are independent.

Let $V$ and $W$ be two linear orders, minimal over some $A$. Then by \cite[Lemma 3.20]{Si} one of the following (mutually exclusive) statements holds:

\begin{itemize}
\item $V$ and $W$ are intertwined;
\item $V$ and the reverse of $W$ are intertwined;
\item $V$ and $W$ are independent.
\end{itemize}

We can now describe definable subsets of products of independent linear orders.

\begin{definition}
	let $A\subseteq M$ and let $V_0,\ldots,V_{m-1}$ be pairwise independent $A$-definable linear orders. For $n<\omega$, an \emph{$A$-sector} of $V_0^{k_0}\times \cdots \times V_{m-1}^{k_{m-1}}$ is a finite union of sets of the form $D_0\times \cdots \times D_{m-1}$, where each $D_i$ is an $A$-sector of $V_{i}^{k_i}$.
\end{definition}

\begin{fact}[\cite{Si}, Proposition 3.34] \label{fact: definable set product} Let $V_0,\ldots,V_{n-1}$ be pairwise independent minimal $A$-definable orders and let $A\subseteq B$. Let $X\subseteq V_0^{k_0}\times \cdots \times V_{n-1}^{k_{n-1}}$ be a $B$-definable set. Then the closure of $X$ is a $B$-sector of $V_0^{k_0}\times \cdots \times V_{n-1}^{k_{n-1}}$.
\end{fact}

All those notions adapt to circular orders instead of linear ones. We will not use circular orders much in this paper, but they will appear through Fact \ref{fact:gluing} below.

Let $(V,C)$ be an $A$-definable circular order, where $C(x,y,z)$ is the circular order relation. We assume that the order is dense. If $a\in V$ is any point, then we can form the linear order $V_{a\to} = V\setminus \{a\}$ equipped with the linear order inherited from $C$. If $a,b\in V$ are two points, we can naturally identify a cut in the completion of $V_{a\to}$ to a cut in the completion of $V_{b\to}$, except for the cut corresponding to $b$ which has been sent at infinity in $V_{b\to}$. We then define the completion $\overline V$ of $V$ to be the set of all cuts of $V_{a\to}$, plus $a$ itself. With the identification described in the previous sentence, this does not depend on $a$ and is $\bigvee$-definable over the same set $A$ as $V$.

\begin{defi}
The order $(V,C)$ is minimal over $A$ if:
\begin{itemize}
\item the order is dense;
\item no element of the completion $\overline V$ is algebraic over $A$;
\item there is no definable (over any set of parameters) convex equivalence relation with infinitely many infinite classes.
\end{itemize}
\end{defi}

As for linear orders, the third condition is automatically satisfied if $V$ has rank 1.

We define sectors as for linear orders and obtain a similar description of closures of definable sets.

\begin{definition}
	Let $A\subseteq M$ and let $(V,C)$ be an $A$-definable circular order. For $n<\omega$, an \emph{$A$-sector} of $V^n$ is a subset of $V^n$ defined by a formula $\phi(x_0,\ldots,x_{n-1})$ which is a finite boolean combination of relations of the form:
	
	\begin{itemize}
		\item $x_i = x_j$, for $i,j<n$;
		\item $C(x_i, x_j, x_k)$, for $i,j,k<n$;
		\item $C(a, x_i,x_j)$, for $i,j<n$ and $a\in \acl(A)\cap \overline V$;
		\item $x_i \in W$, for $i<n$ and $W$ an $\acl(A)$-definable convex subset of $V$.
	\end{itemize}
\end{definition}

\begin{definition}
	Let $A\subseteq M$ and let $V_0,\ldots,V_{m-1}$ be pairwise independent $A$-definable linear or circular orders. For $n<\omega$, a \emph{sector} (resp. \emph{$A$-sector}) of $V_0^{k_0}\times \cdots \times V_{m-1}^{k_{m-1}}$ is a set of the form $D_0\times \cdots \times D_{m-1}$, where each $D_i$ is a sector (resp. $A$-sector) of $V_{i}^{k_i}$.
\end{definition}

\begin{fact}[\cite{Si}, Corollary 4.15]\label{fact:all order types with parameters}
	Let $V_0,\ldots,V_{n-1}$ be pairwise independent, minimal $\emptyset$-definable circular orders. Let $V_n,\ldots V_{m-1}$ be pairwise independent minimal $\emptyset$-definable linear orders. Let $X\subseteq V_0^{k_0}\times \cdots \times V_{m-1}^{k_{m-1}}$ be definable over some parameters $A$. Then the topological closure of $X$ is an $A$-sector of $V_0^{k_0}\times \cdots \times V_{m-1}^{k_{m-1}}$.
\end{fact}


\subsection{\th-rank one sets, minimal orders and indiscernible sets.}
Recall that a structure is \emph{primitive} if it does not admit a non-trivial $\emptyset$-definable equivalence relation. The analysis in this paper is based on the classification of primitive \th-rank one dependent $\omega$-categorical sets
which was achieved in the stable case by \cite{CHL} and in the unstable set in \cite{Si}. We now present those results starting with the stable case.

Following \cite{CHL}, we define:

\begin{defi}
  A structure $M$ is said to be \emph{coordinatized by a $\emptyset$-definable set $X$} if
  for any $a\in M$, $\acl(a)\cap X\neq \emptyset$.
\end{defi}

Recall that a definable set $X$ in an $\omega$-categorical structure is \emph{strongly minimal} if every definable subset of $X$ is finite or cofinite. The set $X$ is \emph{strictly minimal} if it is strongly minimal and primitive. It is \emph{almost strictly minimal} if it admits a $\emptyset$-definable equivalence relation such that each class is strictly minimal.

\begin{fact}[\cite{CHL} Theorem 4.1]
Let $M$ be a primitive stable \th-rank 1 $\omega$-categorical structure. Then there is a $\emptyset$-definable almost strictly minimal set $X$ which coordinatizes $M$.
\end{fact}

A strictly minimal $\omega$-categorical set is either an indiscernible set (that is an infinite set with no structure), or an affine or projective space over a finite field. See \cite[Theorem 2.1]{CHL} for a more precise statement. Importantly for us, if $M$ is finitely homogeneous, then any strictly minimal set is an indiscernible set. This can be seen either using Macpherson's result \cite{Mac} that a finitely homogeneous structure does not interpret an infinite group, thus ruling out the other possibilities, or by \cite[Section 8]{Lac84} which gives a direct proof. As we shall see in the following results, there are a lot of analogoues between minimal orders in superrosy structures of finite dp-dimension (see Definition \ref{def:op-dim}) and indiscernible sets. This is the reason why assuming finite homogeneity will remove much of the complexity in the proofs.

\medskip

In the unstable case, we have the following proved in \cite{Si} (see for instance the beginning of Section 6.4).

 \begin{fact}\label{circular}
 Let $M$ be an $\omega$-categorical, primitive, \th-rank one, dependent, unstable structure. Then there is an $\emptyset$-interpretable set $W$, which
 is a finite union of circular and linear minimal orders, any two of which are either independent or in order-reversing bijection, such that $W$ admits a finite-to-one map to $M$.
  \end{fact}
%
%

A starting point to obtain this $W$ is the result of \cite{Si1} that gives us linear orders in an unstable dependent structure. To state it, we first need to define the op-dimension.

\begin{definition}
	An ird-pattern of length $\kappa$ for the partial type $\pi(x)$ is given by:
	
	\begin{itemize}
		\item a family $(\phi_\alpha(x;y_\alpha):\alpha<\kappa)$ of formulas;
		\item an array $(b_{\alpha,k}:\alpha<\kappa,k<\omega)$ of tuples, with $|b_{\alpha,k}|=|y_\alpha|$;
	\end{itemize}
	such that for any $\eta\colon  \kappa \to \omega$, there is $a_\eta\models \pi(x)$ such that for any $\alpha<\kappa$ and $k<\omega$, we have \[ \models \phi_\alpha(a_\eta;b_{\alpha,k}) \iff \eta(\alpha)<k.\]
\end{definition}

\begin{definition}\label{def:op-dim}
	We say that $T$ has op-dimension less than $\kappa$, and write $\dlr(T)<\kappa$ if, in a saturated model of $T$, there is no ird-pattern of length $\kappa$ for the partial type $x=x$.
\end{definition}

\begin{fact}\label{Theorem 5.13}\label{unstable formula}
Let $T$ be $\omega$-categorical, dependent, of
op-dimension at least $l$. Then there is a finite set $A$, an infinite
$A$-definable transitive set $X$, and $l$ definable linear quasi orders
$\leq_{1,A}, \dots, \leq_{l,A}$ on $X$ such that the structure $(X,\leq_{1,A}, \dots, \leq_{l,A})$
contains an isomorphic copy of every finite structure $(X_0,\leq_{1}, \dots, \leq_{l})$ equipped with $l$
linear orders.

In particular, there are equivalence relations
$E_{1, A}, \dots, E_{l,A}$ on $X$ such that $(X/E_{i,A}, \leq_{i,A})$ are infinite dense independent
linear orders without endpoints.
\end{fact}

We now recall some facts about indiscernible sets.  The first is Lemma 2.4 in \cite{CHL}, adapting the definitions to the trivial case.

\begin{fact}\label{2.5}
Let $H$ be a 0-definable indiscernible set. Let $A$ be any set, and let $H/A$ be the localization (as defined in Section 2 of \cite{CHL}). Then
\[
H/A=H/(acl(A)\cap H)=H\setminus \acl(A).
\]
\end{fact}

The following is Corollary 2.5 in \cite{CHL} after adding constants to the language.

\begin{fact}\label{2.4}
Let $H_c$ and $H_d$ be non-orthogonal indiscernible sets, definable over tuples $c$ and $d$ respectively. Then there is a unique $\acl(cd)$-definable bijection $f_{cd}:H_c\setminus\acl(cd)\rightarrow H_d\setminus\acl(cd)$. For any $x\in H_c\setminus\acl(cd)$ this bijection is defined by $f_{cd}(x)=y$ if and only if $y\in \acl(cdx)$.
\end{fact}

We can now prove some key results which hold for both indiscernible and minimally ordered sets.

\begin{fact}\label{6.4}\label{fact:gluing}\label{prop:family of orders}

Let $\{X_d\}_{d\in D}$ be a $\emptyset$-definable family of definable sets which  are either all minimal linearly ordered, all minimal circularly ordered, or all indiscernible. Then there is a $\emptyset$-interpretable set $F$ and a $\emptyset$-definable family $\{W_e\}_{e\in F}$ such that:
\begin{itemize}
\item For $e\in F$, $W_e$ is either a minimal linear or circular order, or an indiscernible set.
\item For any pair of tuples $e_1\neq e_2\in F$ the sets $W_{e_1}$ and $W_{e_2}$ are orthogonal in the indiscernible case, and in the order case, they are either independent or in order-reversing bijection.
\item For any $d\in D$ there is a unique $e\in F$ such that $X_d$ admits a definable injection into $W_e$ and in the order case, this injection is order-preserving.
\end{itemize}
\end{fact}

\begin{proof}
In the ordered case this is Theorem 6.4 in \cite{Si}. The strongly minimal case is known, and the proof is implicit in \cite{CHL}. We include it for completeness.

Let $R$ be the definable equivalence relation on $D$ defined by $d_1Rd_2$ if and only if $X_{d_1}$ and $X_{d_2}$ are non orthogonal. Let $E$ be a relation in $\bigcup_{d\in D} X_d$ be defined by $x_1 E x_2$ if and only if $x_1\in X_{d_1}\setminus \acl(d_1d_2)$, $x_2\in X_{d_2}\setminus \acl(d_1d_2)$, for some $d_1 R d_2$ and $f_{d_1d_2}(x_1)=x_2$ (where $f$ is the function given in Fact \ref{2.5}). We will prove that $E$ is an equivalence relation, which implies that all sets non orthogonal to a fixed $X_d$ inject into
\[
\left(\bigcup_{d' R d} X_{d'}\right)/E
\]
and the lemma holds with $F:=D/R$.

To prove that $E$ is an equivalence relation we need to prove transitivity. We first prove transitivity under some independence assumptions.

\begin{claim}
Let $a,b,c\in D$, let $x\in X_a\setminus \acl(ab)$, and assume that $c\thind x,a,b$. Then $f_{bc}(f_{ab}(x))=f_{ac}(x)$.
\end{claim}
\begin{claimproof}
The image of any element in $X_a\setminus \acl(abc)$ is $abc$-interalgebraic with its image under $f_{ab}$ so by definition it must be mapped
to an element in $X_b\setminus \acl(abc)$. It follows that both $f_{bc}\circ f_{ab}$ and $f_{ac}$ are maps from $X_a\setminus \acl(abc)$ to $X_c\setminus \acl(abc)$, and by Fact \ref{2.4} they must coincide.

By hypothesis $x\not\in \acl(ab)$ and by \th-independence (and monotonicity) $c\thind_{ab} x$ so $x\not\in \acl(abc)$. The result follows.
\end{claimproof}

Let $X_a, X_b$ and $X_c$ be non orthogonal strongly minimal sets ($a,b,c\in D$) and let $x$ be any element in $X_a$, $y\in X_b$ and $z\in X_c$. Assume that $x E y$ and $y E z$ so that $f_{ab}(x)=y$ and $f_{bc}(y)=z$ with $x\not\in \acl(ab)$ and $y\not\in \acl(bc)$.

Let $d\in D$ be such that $d\thind a,b,c,x$ with $X_d$ non orthogonal to $X_a, X_b$ and $X_c$. By the claim $f_{bd}(f_{ab}(x))=f_{ad}(x)$, $f_{cd}(f_{ac}(x))=f_{ac}(x)$ and $f_{cd}(f_{bc}(y))=f_{bd}(x)$. It follows that
\[
f_{cd}(f_{bc}(f_{ab}(x)))=f_{ad}(x)=f_{cd}(f_{ac}(x)).
\]

But in its domain $f_{cd}$ is a bijection so $f_{bc}(f_{ab}(x)))=f_{ac}(x)$. By definition $z=f_{bc}(y)=f_{bc}(f_{ab}(x))$ so $x E z$, as required.
\end{proof}

\begin{lemma}\label{emptyset definable}\label{emptysetdefinable}
Let $\{W_d\}_{d\in D}$ be a definable family of \th-rank one sets satisfying the conditions in the conclusion of Fact \ref{6.4}. Let $b$ be any element (in $M$) such that for some $d\in D$ we have $b\thind d$ and there is $\alpha \in W_d\cap \acl(bd)$. Then $d\in \acl(\emptyset)$ and $\alpha\in \acl(b)$.
\end{lemma}

\begin{proof}
Assume otherwise, then by independence, we have $d\notin \acl(b)$. Let $( d_i)_{i\in \omega}$ be a \th-Morley sequence over $b$ with $d=d_0$, and let $\alpha_i$ be such that $( d_i \alpha_i)_{i\in \omega}$ is a $b$-indiscernible sequence with $d_0\alpha_0=d\alpha$.

\medskip \noindent
\underline{Claim:}
For some $i$ we have $\alpha_{i}\in \acl(\alpha_{<i} d_{<i} d_i)$.
\medskip

\begin{claimproof}
By finiteness of rank, for some $i$ we have $\thr(b/\alpha_{<i}d_{<i})=\thr(b/\alpha_{\leq i}d_{\leq i})$. This implies that $b\thind_{\alpha_{<i}d_{<i}} \alpha_id_i$ and by monotonicity $b\thind_{\alpha_{<i}d_{\leq i}}  \alpha_i$. Since $\alpha_i\in \acl(d_ib)$ this implies $\alpha_i\in \acl(\alpha_{<i} d_{\leq i})$, as required.
\end{claimproof}

Let $\bar d:=( d_1, d_2, \dots)$. Notice that $\alpha_i\not\in \acl(\bar d)$ (otherwise $\alpha_i$ would witness $b\nthind_{d_i} \bar d$, contradicting our hypothesis). Let $n$ be the first element for which $\alpha_n\in \acl(\alpha_{<n}\bar d)$, which exists by the claim.

Let $m<n$ be the first natural number such that $\alpha_n\in \acl(\alpha_0\dots \alpha_m, \bar d)$. The two types $\tp(\alpha_m/ \bar d \alpha_1\dots \alpha_{m-1})$ and $\tp(\alpha_{n}/ \bar d \alpha_1\dots \alpha_{m-1})$ have $\thr$-rank 1, so we have interalgebraicity between infinite subsets of $W_{d_m}$ and $W_{d_n}$.

If the $W_{d}$'s are strongly minimal, this contradicts our hypothesis of orthogonality. In the order case, then the above interalgebraicity induces a bijection between infinite subsets of $W_{d_m}$ and $W_{d_n}$ which in turn induces independent orders on $\tp(\alpha_m/ \bar d \alpha_1\dots \alpha_{m-1})$. But $( d_i)_{i\in \omega}$ is an indiscernible sequence so we would be able to define arbitrarily many independent orders on $\tp(\alpha_m/ \bar d \alpha_1\dots \alpha_{m-1})$, contradicting that the theory is dependent.
\end{proof}

We end with the following triviality property which is a generalization of \cite[Proposition 6.1]{Si}.

\begin{prop}\label{6.1}
Let $X$ be either a $\emptyset$-definable minimal linear or circular order, or a $\emptyset$-definable indiscernible set. Suppose that $a\thind b$ are any two tuples. If $c\in \acl(ab)\cap X$, then $c\in \acl(a)\cup \acl(b)$.
\end{prop}

\begin{proof}
The statement for linear and circular orders is precisely Proposition 6.1 of \cite{Si}. The proof for indiscernible sets is essentially the same:

If $a\in \acl(b)$ then they are both in $\acl(\emptyset)$ and there is nothing to show.

Assume $a\not \in \acl(b)$ and let $a_1,a_2,\dots$ be a \th-Morley sequence (over $\emptyset$) of $\tp(a/b)$ over $b$. Let $c_i$ be such that $\tp(c_ia_i/b)=\tp(ca/b)$. Assume towards a contradiction that $c_i\not\in \acl(a_i)\cup \acl(b)$, so that by \th-independence $c_i\neq c_j$ for any $i\neq j$. Let $\overline{a}= a_1,a_2,\dots, a_N$. By construction $b\thind \overline a$ so $b\thind_{a_i} \overline a$, $bc_i\thind_{a_i} \overline a$ and $c_i\thind_{a_i} \overline a$ by monotonicity. In particular, $c_i\not\in \acl(\overline a)$.

Let $X_{\overline a}=X/{\acl(\overline a)}=X\setminus \acl({\overline a})$ be the localization of
$X$ over $\overline a$. But then $c_1,\dots, c_N$ is an $N$-dimensional set in the algebraic closure of $b$. Since $\thr(b/\emptyset)=\thr(b/\overline a)$ is fixed and $N$ can be any natural number, this contradicts $\omega$-categoricity.\end{proof}

\subsection{Distality}

Distality was introduced in \cite{distal}. It is meant to capture the notion of a purely unstable NIP structure.

\begin{definition}\label{def:distal}
	A structure $M$ is called \emph{distal} if for every formula $\phi(x;y)$, there is a formula $\psi(x;z)$ such that for any finite set $A\subseteq M$ and tuple $a\in M^{|x|}$, there is $d\in A^{|z|}$ such that $\psi(a;d)$ holds and for any instance $\phi(x;b)\in \tp(a/A)$, we have the implication \[M\models (\forall x) \psi(x;d)\to \phi(x;b).\]
\end{definition}

Assume that $M$ is finitely homogeneous. Then if $M$ is distal, there is an integer $k$ such that for any finite set $A$ and singleton $a\in M$, there is $A_0\subseteq A$ of size $\leq k$ such that $\tp(a/A_0)\vdash \tp(a/A)$. (That is, if $\tp(a'/A_0)=\tp(a/A_0)$, then $\tp(a'/A)=\tp(a/A)$.)

For us, the main reason for being interested in distality is the following result from \cite{Si}

\begin{fact}[\cite{Si} Theorem 8.3]\label{fact:distal fin axiom}
A distal finitely homogeneous structure is finitely axiomatizable.
\end{fact}

In Section \ref{sec:distal exp}, we will show that a finitely homogeneous dependent rosy structure has a distal expansion which is also finitely homogeneous. As we construct this expansion, we will need to make sense of a definable set being distal, as opposed to the whole structure. All that we need is summarized in a series of facts below. However, since they do not all appear explicitly in the literature, we give some details.

We first define distality for indiscernible sequences following \cite{distal}. An indiscernible sequence $I+J+K$ ($I,J$ and $K$ densely ordered without endpoints) is distal if for any tuples $a$ and $b$, if $I+a+J+K$ and $I+J+b+K$ are indiscernible, then so is $I+a+J+b+K$. A structure is distal if and only if all indiscernible sequences are distal. We will say that a partial type $\pi(x)$ is distal if every indiscernible sequence in $\pi$ is distal. If $\pi(x)$ is a formula, this is equivalent to saying that the structure with universe $\pi(M)$ equipped with the induced structure from $M$ is distal.

The following fact is implicit in \cite{distal}, where it is proved that a theory is distal if and only if all 1-types are distal, but does not seem to appear anywhere in the literature, so we give some explanations below.

\begin{fact}\label{fact:distal trans}
 If $\tp(a/A)$ and $\tp(b/Aa)$ are distal, then $\tp(ab/A)$ is distal.
 \end{fact}

The main point of the proof of reduction to dimension 1 in \cite{distal} is the following property.

\begin{lemma}\label{lem:co-distal}
Assume that $\pi(x)$ is a partial type over some $A$ which is distal and let $I+J$ be any sequence in $M$. Let $a$ be a tuple from $M$ and $b$ a tuple of realizations of $\pi(M)$. Assume that $I+J$ is indiscernible over $Ab$ and $I+a+J$ is indiscernible over $A$. Then $I+a+J$ is indiscernible over $Ab$.
\end{lemma}
\begin{proof}
This is the content of \cite{distal}, but does not appear explicitly there. It is sated explicitly along with a complete proof in \cite[Proposition 1.17]{distalValFields}.
\end{proof}

We state some form of converse.

\begin{lemma}\label{lem:distal co-distal}
Let $\pi(x)$ be a partial type over some $A$. Assume that for any $I+J$ be any sequence in $M$, tuple $a$ from $M$ and $b$ a realization of $\pi$, if $I+J$ is indiscernible over $A$ and $I+a+J$ is indiscernible. then $I+a+J$ is indiscernible over $Ab$. Then $\pi(x)$ is distal.
\end{lemma}
\begin{proof}
We show directly that an indiscernible sequence in $\pi(x)$ is distal. Let $I+J+K$ be such a sequence and let $a,b$ be given so that $I+a+J+K$ and $I+J+b+K$ are indiscernible. Let $I',J'$ be the sequences formed by concatenating a fixed enumeration of $K$ to each element of the sequences $I$ and $J$ respectively. Let also $a'$ be obtained from $a$ by concatenating the same enumeration of $K$. Then $I'+a'+J'$ is indiscernible over $A$. Applying the hypothesis, we get that $I'+a'+J'$ is indiscernible over $Ab$, hence $I+a+J+b+K$ is indiscernible as required.
\end{proof}

We can now give the proof of Fact \ref{fact:distal trans}. Assume that $\tp(a/A)$ and $\tp(b/Aa)$ are distal. Let $I+J$ be an indiscernible sequence from $M$ and $c\in M$ so that $I+c+J$ is indiscernible over $A$ and $I+J$ is indiscernible over $Aab$. As $\tp(a/A)$ is distal, we get by Lemma \ref{lem:co-distal} that $I+c+J$ is indiscernible over $Aa$. Then as $\tp(b/Aa)$ is distal, $I+c+J$ is indiscernible over $Aab$. By Lemma \ref{lem:distal co-distal}, $\tp(ab/A)$ is distal.

\begin{fact}\label{fact:distal acl}
If the definable set $D$ is distal and $D'$ is a definable subset of $\acl(D)$, then $D'$ is distal.
\end{fact}
\begin{proof}
See \cite[Corollary 1.26]{distalValFields}.
\end{proof}

\begin{fact}\label{fact:distal orth}
Let $\tp(a/A)$ be distal and $X$ an $A$-definable stable set. Then $\acl(Aa)\cap X = \acl(A)\cap X$.
\end{fact}
\begin{proof}
Given Fact \ref{fact:distal acl}, it is enough to show that a type which is both distal and stable is algebraic. Indeed if $\pi(x)$ is stable and non-algebraic, let $I+a+J+K$ be an indiscernible non-constant sequence of realizations of $\pi$. Then since $\pi$ is stable, the sequence is totally indiscernible. Hence $I+J+a+K$ is also indiscernible. However $I+a+J+a+K$ is not indiscernible, so this contradicts distality.
\end{proof}

Finally, from \cite{Si}, we have (though the ``in particular'' part is obvious):

\begin{fact}[\cite{Si} Theorem 1.3]\label{fact:distal rank 1}
Any \th-rank 1 unstable set is distal. In particular a pure dense linear order is distal.
\end{fact}

\section{Finiteness of the rank}\label{sec:finite rank}

For the rest of the paper, we will assume taht we are working in a finitely homogeneous dependent rosy structure, although some results hold in more generality. Let $\mathcal L$ be a finite relational language in which $M$ has quantifier elimination.

We start by showing that any rosy finitely homogeneous structure has finite \th-rank.

\begin{lemme}
Let $(a_i:i<\kappa+\kappa)$ be an indiscernible sequence in a rosy theory, where $\kappa \geq |T|^+$. Then $(a_i:\kappa \leq i <\kappa+\kappa)$ is a Morley sequence over $(a_i:i<\kappa)$.
\end{lemme}
\begin{proof}
Assume that say $a_{\kappa+1} \nthind_{a_{<\kappa}} a_{\kappa}$ and let $\phi(x;a_{\kappa})$ be a formula with hidden parameters from $a_{<\kappa}$ that \th-forks over $a_{<\kappa}$. Assume that the hidden parameters of $\phi$ are in $a_{<\alpha}$ for $\alpha<\kappa$. Then for any $\alpha\leq \beta<\kappa$, the formula $\phi(x;a_{\beta})$ \th-forks over $a_{<\beta}$. But all those formulas are satisfied by $a_{\kappa}$. This gives a sequence of \th-forking extensions of a type of length $\kappa$ and this contradicts rosyness.
\end{proof}

\begin{prop}\label{prop:finite rank homogeneous}
If $M$ is finitely homogeneous and rosy, then it has finite $\thr$-rank.
\end{prop}
\begin{proof}
Let $\kappa =\aleph_1$. The hypothesis of finite homogeneity has the following consequence: there are only finitely many types over $\emptyset$ of indiscernible sequence $(a_i:i<\kappa)$ of singletons. Let $n$ be that number of types of indiscernible sequences. We show that $\thr(M)\leq n$. Assume for a contradiction that we can find an increasing sequence of sets $(A_i:i<n+1)$ and a type $p\in S_1(A_n)$ such that for every $k<n+1$, $p_k := p|_{A_k}$ \th-forks over $A_{k-1}$ (where we set $A_{-1}=\emptyset$). For each $k$, let $\bar a_k = (a_{k,i}:i<\kappa+\kappa)$ be a Morley sequence in $p_k$ over $A_n$. By assumption, there are $k<l<n+1$ such that the sequences $\bar a_k$ and $\bar a_l$ have the same type over $\emptyset$. Let $\sigma$ be an automorphism of $M$ sending $\bar a_k$ to $\bar a_l$. Then $p_l$ and $\sigma(p_k)$ are both non-\th-forking over $\bar a_l$ and have the same restriction to it.

Recall from \cite{On} that there are local ranks $\thr(p,\Delta,\Pi,k)$ defined for any type $p$, integer $k$ and finite sets of formulas $\Delta$ and $\Pi$ with the property that a type $p\in S(B)$ \th-forks over $A\subseteq B$ if and only if for some $\Delta,\Pi,k$ we have $\thr(p,\Delta,\Pi,k)<\thr(p|_A,\Delta,\Pi,k)$. We for any $\Delta,\Pi,k$ we have \[\thr(p_k,\Delta,\Pi,k)=\thr(\sigma(p_k),\Delta,\Pi,k) = \thr(\sigma(p_k)|_{\bar a_l},\Delta,\Pi,k) = \thr(p_l|_{\bar a_l}, \Delta,\Pi,k)= \thr(p_l,\Delta,\Pi,k).\] Hence the extension $p_k \subseteq p_l$ is not \th-forking; contradiction.
\end{proof}

\subsection{Coordinatization}

We will now prove that every finitely homogeneous rosy dependent
 structure is coordinatized by a \th-rank 1 $\emptyset$-definable set.

\begin{theorem}\label{th:coord}
Let $M$ be a dependent, finitely homogeneous, rosy structure. Then $M$ is coordinatized by a \th-rank 1 formula.
\end{theorem}
\begin{proof}
Let $M$ be finitely homogeneous, dependent and rosy. By Proposition \ref{prop:finite rank homogeneous} it has finite \th-rank,
equal to some $n$. Without loss of generality we may assume that $M$ is transitive (has a unique 1-type over $\emptyset$).

We will prove Theorem \ref{th:coord} by induction on $n=\thr(M)$. To this end, assume that for any transitive $A$-definable subset $X$ of \th-rank less than $n$ there is
a \th-rank one transitive $A$-definable set $Y$ such that for any $b\in X$ we have $\acl(bA)\cap Y\neq \emptyset$.

In particular, it is enough to prove that $M$ is coordinatized by a
formula of \th-rank $<n$.

By definition of \th-rank, there is some definable set $\phi(x,d)$
of \th-rank $n-1$, which therefore \th-forks over $\emptyset$. By
definition of \th-forking (replacing $\phi(x,d)$ by a \th-rank $n-1$ set in the
disjunction of \th-dividing formulas it implies), we may assume that it
\th-divides over $\emptyset$.

By definition of \th-dividing there is some
$a$ such that $\phi(x,d)$ strongly divides over $a$, that is, $\{\phi(x,d')\}_{d'\models \tp(d/a)}$ is $k$-inconsistent for some $k$. This depends only on $\phi(x,d)$ and $\tp(d/a)$, so we may choose $a$ with $b\thind_d a$.

Recall also that $\phi(x,d)$ strong divides over $a$ if and only if
\begin{equation}\label{str dividing}
d\in \acl(ba)\setminus \acl(a)\end{equation} for any realization $b$ of $\phi(x,d)$.

Now, we can increase $a$ so that $\tp(d/a)$ has \th-rank 1 (and still $b\thind_d a$).
We therefore have $d\in \acl(ab)\setminus \acl(a)$, $\tp(b/ad)$ of \th-rank $n-1$ and $\tp(d/a)$ of \th-rank 1.
It follows by Lascar's Inequality (Fact \ref{Lascar}) that $\thr(b/a)=n$ which implies $b\thind a$.

\medskip

We now divide our construction into two cases.

\smallskip
\noindent \textbf{\underline{Case 1}}: The definable set
$\tp(d/a)$ is unstable.

\smallskip

Let $X_a:=\tp(d/a)$. By Fact \ref{Theorem 5.13}, up to increasing $a$, there is an equivalence relation $E_{a}$ on $X_a$ with finitely many classes and an $a$-definable order $\leq_{a}$ on $V_{a}:=X_a/E_{a}$. The set $X_a$ is transitive over $a$, so $(V_{a}, \leq_{a})$ is minimal.

Let $\pi:X_a\to V_a$ be the canonical projection. By hypothesis $\pi(d)\in \acl(ba)$ and since $V_a$ is linearly
ordered, $\pi(d)\in \dcl(ba)$.

Let $p(x; d,a)=\tp(b/da)$.

Applying Fact \ref{fact:gluing} to the family $\{(V_{a'}, \leq_{a'})\}_{a'\models \tp(a/\emptyset)}$, let $(W_{e'}:e'\in F)$
be the corresponding family of orders, pairwise independent or in order-reversing bijection. Let $e\in F$ be such that $V_{a}$ embeds
densely in a convex subset of $W_{e}$. Let $\beta$ be the image of $\pi(d)$ under this embedding.



Now $\beta\in \acl(abe)$ and $\beta\not \in \acl(ae)$ ($X_a$ is transitive and
the image of $V_a$ in $W_e$ is infinite). We have $b\thind a$ and $e\in \acl(a)$ so $b\thind_e a$ and by Lemma \ref{emptysetdefinable}, $\beta\in \acl(eb)$.

By Fact \ref{prop:family of orders} $e\in \acl(b)$ which by \th-independence implies $e\in \acl(\emptyset)$, and $\beta \in \acl(eb)=\acl(b)$. This implies that
\[\thr(b/\beta)\geq \thr(b/dae \beta) =\thr(b/da)=n-1, \] and
$\thr(\beta)\geq 1$.

We then have \[ n=\thr(b,\beta)=\thr(b/\beta)+\thr(\beta) \geq n-1
+1,\] hence $\thr(b/\beta)=n-1$, $\thr(\beta)=1$ and $\tp(\beta)$
is a coordinatizing set of \th-rank 1.

\bigskip

\noindent \textbf{\underline{Case 2}}: The definable set
$\tp(d/a)$ is stable.

\smallskip

Let $\theta(y,a)= \tp(d/a)$. By Observation \ref{U=Uth}, $\theta(y,a)$ has finite Morley rank. By restricting
the formula and taking a quotient by an equivalence relation with finite classes (replacing $\phi(x,d)$ with a finite cover), we can assume that it is strictly minimal. By finite homogeneity, we can increase $a$ further (adding a constant from $\theta(y,a)$ if needed) to have $\theta(y,a)$ an indiscernible set.

We can now apply Fact \ref{6.4} to $\{\theta(y,a')\}_{a'\models \tp(a/\emptyset)}$. Let $\{W_{e'}\}_{e'\in F}$ be the family of strongly minimal orthogonal sets, let $e,\delta$ be such that $e\in \acl(a)$ and $d$ maps to $\delta$ via the injection from $\theta(y,a)$ into $W_e$.

\begin{claim}
$\delta\in \acl(be)$.
\end{claim}

\begin{claimproof}
Since $a\thind b$ and $e\in \acl(a)$ we have $b\thind_e a$. By construction $\delta\not\in \acl(dae)\setminus \acl(ae)$, and since $d\in \acl(ba)$ we have $\delta\not\in \acl(bae)\setminus \acl(ae)$. By Proposition \ref{6.1} $\delta\in \acl(be)$.
\end{claimproof}

\begin{claim}
$e\in \acl(\emptyset)$.
\end{claim}

\begin{claimproof}
Suppose otherwise, fix some integer $N>n=\thr(b)$ and let $\bar e:=( e_i)_{i\leq N}$ by a \th-Morley sequence of $\tp(e/b)$. Since $\delta_i\in\acl(be_i)$ and $b\thind_{e_i} \bar e$ we have $\delta_i\thind_{e_i} \bar e$ so
that $\delta_i$ is a generic element in $W_{e_i}/\bar e$.

Since $\{W_{e_i}\}_{i\leq N}$ are orthogonal by hypothesis, $\delta_{j+1}\not \in \acl(\delta_1\dots\delta_{j}\bar e)$ so that $\thr(\delta_1\dots\delta_{N}/\bar e)\geq N$. It follows that $\thr(b/\bar e)\geq N>n$, a contradiction.
\end{claimproof}

This implies that $\tp(b)$ is coordinatized by $\tp(\delta/\acl(\emptyset))$, as required.
\end{proof}

\begin{corollary}\label{cor:coordinatization}
Let $M$ be a dependent, finitely homogeneous rosy structure. Then for any $b\in M$ there is a sequence $\{\beta_0, \beta_1, \beta_2, \dots, \beta_n\}$ such that
\begin{itemize}
\item $\beta_i\in \acl(b), \beta_0=\emptyset$ and $b\in \acl(\beta_n)$.
\item Either $\tp(\beta_{i+1}/\beta_i)$ admits a $\beta_i$-definable minimal order or is indiscernible.
\end{itemize}
\end{corollary}

\begin{proof}
The proof is by induction on the \th-rank of $\tp(b/\emptyset)$.

By Theorem \ref{th:coord}, there is some \th-rank one set $Y$ and some $d\in Y$ in $\acl(b)$. If $Y$ is unstable, then by Fact \ref{circular} there is some order $V$ and some $\beta_1\in V$ interalgebraic
with $d$. On the other hand, if $Y$ is stable, it has Morley Rank one and there is some strictly minimal $X$ and some $\beta_1\in X$ interalgebraic with $d$. In either case, $\beta_1$ is set.

For the induction step, consider $\tp(b/\beta_1)$. It has \th-rank less than $\thr(b/\emptyset)$, so adding $\beta_1$ to the language, by induction hypotesis there are $\{\beta_1', \dots, \beta_l'\}$ with
$\beta_i'\in \acl(b)$ and $b\in \acl(\beta'_n\beta_1)$ and $\tp(\beta_{i+1}'/\beta_i)'$ admits a $\beta_i'\beta_1$-definable minimal order or it is stable and strictly minimal. If we define for $i>1$  $\beta_{i+1}=\beta_i'\beta_1$, the sequence $\{\beta_0, \beta_1, \beta_2, \dots, \beta_l+1\}$ satisfies the conditions of the corollary.
\end{proof}

\begin{corollary}\label{cor:two-types}
	If $\rk(M)$ is finite, greater than 1, then $M$ has at least two 2-types of distinct elements.
\end{corollary}
\begin{proof}
	This is clear if $M$ is not primitive. Assume that $M$ is primitive and let $V$ be a rank 1 coordinatizing set. Let $V(a)=\dcl(a)\cap V$. The relation $V(a)=V(a')$ defines an equivalence relation on $M$. By primitivity, it is trivial. As $M$ has rank $>1$, this implies that $V(a)$ has more than one element. Then we can find $a\neq a' \in M$ such that $\dcl(a)\cap \dcl(a')\cap V$ is non-empty. We can also find $a\neq a'\in M$ for which this intersection is empty. Hence $M$ has more than one 2-type.
\end{proof}

\section{The geometry of a finitely homogeneous dependent structure of finite \th-rank}

\begin{theorem}\label{modularity}\label{triviality}
Let $a,b$ be tuples in $M$. Then the following hold.
\begin{enumerate}
\item (Modularity:) \[a\thind_{\acl^{eq}(a)\cap \acl^{eq}(b)}  b.\]

\item (Order triviality:)
If $V$ is an ordered \th-rank one definable set, and
$\overline{V}$ is its completion, then \[\acl(ab)\cap \overline V=\left(\acl(a)\cup \acl(b)\right)\cap \overline{V}.\]

\item (Stable triviality:) If $X$ is a $\emptyset$-definable strongly minimal \th-rank one set, then \[\acl(ab)\cap X=\left(\acl(a)\cup \acl(b)\right)\cap X.\]
\end{enumerate}

\end{theorem}

\begin{proof}
To prove modularity, it is enough to show that if $a\nthind b$ then $\acl(a)\cap \acl(b)\neq \acl(\emptyset)$. The result then follows by induction on $\thr(a/b)$ (adding $\acl(a)\cap \acl(b)$ to the language). The base case is when $a\in \acl(b)$ where the result holds trivially.

So assume that $a\nthind b$.

By Corollary \ref{cor:coordinatization}, we know that $a$ as inter-algebraic with a tuple $( a_i)_{i\leq N}$ such that $\tp(a_{i+1}/a_1\dots a_i)$ has \th-rank one and is either minimal ordered or indiscernible (strictly minimal sets in a finitely homogeneous structure are indiscernible). By $\thr$-rank 1, $a_{i+1}\in \acl(a_1\dots a_ib)$ if and only if $a_{i+1}\nthind_{a_1\dots a_i} b$.

By transitivity of \th-rank and our hypothesis there is a smallest $i$ such that $a_i\in \acl(a_{<i} b)$. By transitivity of \th-independence and minimality of $i$, if $d:= a_{<i}$ then $d\thind b$.

Let $X_d=\tp(a_{i}/d)$ and let $\{W_{e'}\}_{e'\in F}$ be the sets that we get applying Fact \ref{6.4} to the family $\{X_{d'}\}_{d'\models \tp(d/\emptyset)}$. Let $e$ be the parameter of the set into which $X_d$ injects and $\alpha$ the image of $a_{i}$, so that $e\in \acl(d)$ and $\alpha\in \acl(a_id)\subseteq \acl(db)$. By Fact \ref{6.1} (adding $e$ to the language and noticing that $d\thind_e b$) we know that $\alpha\in \acl(be)$.

As $e\in \acl(d)$ and $d\thind b$, we have $e\thind b$ and by Lemma \ref{emptyset definable}, $\alpha\in \acl(b)$. But $\alpha\in \acl(a_id)\subseteq \acl(a)$, which completes the proof of modularity.

\smallskip

Order triviality now follows from (and generalizes) Proposition \ref{6.1}.  Let $A=\acl\left(a\right)\cap \acl\left(b\right)$ so that by modularity $a\thind_A b$. Adding $A$ to the language, $V$ is still an ordered \th-rank one definable set. Proposition \ref{6.1} gives that
\[\acl_A(ab)\cap \overline V=\left(\acl_A(a)\cup \acl_A(b)\right)\cap \overline V.\]
But by definition of $A$ this implies \[\acl(ab)\cap \overline V=\left(\acl(a)\cup \acl(b)\right)\cap \overline V,\] as required.

Stable triviality is proved in the same way.
\end{proof}

\begin{rem}\label{single parameter}\emph{Definability of $\overline{V}$}:
Let $(V,\leq)$ be a definable minimal order. Using triviality of algebraic closure, any element (or cut) in a \th-rank one structure is definable from a single paramater of $M$. Since there are finitely many 1-types (by $\omega$-categoricity) we conclude as in \cite[Corollary 6.2]{Si} that the completion $\overline{V}$ of $V$ is definable.
\end{rem}

\section{Constructing a distal expansion}\label{sec:distal exp}

We prove that any dependent, finitely homogeneous structure $M$ admits an expansion (of the language) that is
finitely homogeneous and distal. The construction of the distal expansion will be done by adding a dense linear order structure to strictly minimal definable subsets of $M$. The following propositions is key for the induction.

\begin{prop}\label{homogeneous coordinatization}
Let $M$ be a $\omega$-categorical, dependent structure of finite \th-rank. Then for any tuple $a\in M$ there is a tuple $(\alpha_0, \dots, \alpha_n)$ of elements in $M^{eq}$ which is interalgebraic with $a$ and such that for every $i$, $\tp(\alpha_{i}/\alpha_{<i})$ is one of the following:
\begin{itemize}
\item algebraic;
\item primitive unstable of \th-rank one;
\item stable, strictly minimal.
\end{itemize}
\end{prop}

\begin{proof}
We prove the result by induction on $\thr(a/\emptyset)$.

By Theorem \ref{th:coord}, let $X$ be a \th-rank one formula coordinatizing $\tp(a/\emptyset)$. Let $\alpha\in X\cap \acl(a)$. Let $F$ be the maximal 0-definable equivalence relation with infinite classes on $\tp(\alpha/\emptyset)$ and let $\alpha_0:=[\alpha]_F$. Because $X$ has \th-rank one, $F$ has finitely many classes and $\alpha_0\in \acl(\emptyset)$.

The type $X_{\alpha_0}:=\tp(\alpha/\alpha_0)$ does not admit any $\alpha_0$-definable equivalence relation with infinite classes. Let $E$ be the maximal $\alpha_0$-definable equivalence relation on $\tp(\alpha/\alpha_0)$. Since every $E$-class is finite, if we let $\alpha_1$ be the image of $\alpha$ under the projection $X_{\alpha_0}\rightarrow X_{\alpha_0}/E$ and $\alpha_2=\alpha$, we have that $\tp(\alpha_1/\alpha_0)$ is primitive and $\tp(\alpha_2/\alpha_1)$ is finite.

Notice that
\[\thr(a/\emptyset)=\thr(a \alpha/\emptyset)=\thr(a/\alpha)+\thr(\alpha/\emptyset)=\thr(a/\alpha)+1.\]

Adding $\alpha$ as a constant to the language we can, by induction hypothesis, find a tuple $(\alpha'_0, \dots, \alpha'_n)$ interalgebraic with $a$ over $\alpha$ satisfying the conditions of the proposition. If we let $\alpha_{i+2}=\alpha'_i$ we get the required sequence $(\alpha_0, \dots, \alpha_{n+3})$.
\end{proof}

\begin{prop}\label{adding R}
Let $M$ be a dependent structure admitting quantifier elimination in a finite relational language $\mathcal L$. Let $H_a$ be a strictly minimal definable subset such that $\thr(a/\emptyset)$ is minimal. Let $p(x):=\tp(a/\emptyset)$, and let $R(x,y;z)$ be a new ternary relation symbol (not in $\mathcal L$). Let $M^R$ be an expansion of $M$ to $\mathcal L\cup \{R\}$ such that:
\begin{itemize}
\item $R(x,y;z)\vdash p(z)\wedge x,y\in H_z$.
\item For any $a'$ with $\tp(a'/\emptyset)=p$, $R(x,y,a')$ defines a dense linear order without endpoints on $H_{a'}$.
\end{itemize}
Then $M^R$ is a finitely homogeneous dependent structure.
\end{prop}

\begin{proof}
Let $M, a, H_a$ and $M^R$ satisfy the conditions of the lemma.

We will need a couple of claims.

\begin{claim}\label{minimal is distal}
The type $p(x)$ is distal.
\end{claim}

\begin{claimproof}
By Proposition \ref{homogeneous coordinatization} there is some sequence $(\alpha_0, \dots, \alpha_n)$ interalgebraic with $a$ such that $\tp(\alpha_{i}/\alpha_{<i})$ is either
algebraic, primitive unstable of \th-rank one, or strictly minimal. If $\tp(\alpha_{i}/\alpha_{<i})$ is infinite, then by Lascar's inequalities \[\thr(\alpha_{<i}/\emptyset)<\thr(\alpha_{\leq i}/\emptyset)\leq \thr(a/\emptyset),\] so by minimality of $\thr(a/\emptyset)$, $\tp(\alpha_{i}/\alpha_{<i})$ is either finite or unstable of \th-rank one, hence distal by \cite{Si}. By Fact \ref{fact:distal trans}, $\tp(\alpha_{\leq n}/\emptyset)$ is distal and by interalgebraicity and Fact \ref{fact:distal acl}, $\tp(a/\emptyset)$ is distal.
\end{claimproof}

\begin{claim}\label{I}Let $a_1, a_2\models p(x)$ and $\bar a$ be a tuple of realizations of $p(x)$. Then the following hold in the $\mathcal L$-structure $M$.
\begin{itemize}
\item For any $a_1, a_2\models p(x)$ and $a_2\not\in \acl(a_1)$, the sets $H_{a_1}$ and $H_{a_2}$ are orthogonal.

\item For any tuple $\bar a$ of elements satisfying $p(x)$, $H_{a}$ is strictly minimal over $\bar a$.
\end{itemize}
\end{claim}

\begin{claimproof}
The first item follows from Fact \ref{6.4}: let $W_e$ be the set in the family of orthogonal strongly minimal sets in which $H_{a_1}$ injects. By minimality of $\thr(a/\emptyset)$, we have $\thr(e/\emptyset)\geq \thr(a/\emptyset)$. Also $e\in \acl(a_1)$ hence $a_1\in \acl(e)$. So $e\not\in \acl(a_2)$ and by construction $H_{a_1}$ and $H_{a_2}$ are orthogonal.

For the second item, we have that $H_{a_1}$ is strictly minimal over $\bar a$ if and only if $\acl (\bar a)\cap H_{a_1}=\emptyset$. By Claim \ref{minimal is distal} $\tp(a'/a_1)$ is distal for any $a'\models p(x)$ so by Fact \ref{fact:distal orth} $\acl(a'a_1)\cap H_{a_1}=\emptyset$.  Fact \ref{6.1} implies that $\acl(\bar a a_1)\cap H_{a_1}= \emptyset$, as required.
\end{claimproof}

\begin{claim}\label{III}
Let $a'\models p(x)$ and let $b$ be any tuple in $M$. If $a'\not\in \acl(b)$ then $\acl(a'b)\cap H_{a'}=\emptyset$.
\end{claim}
\begin{claimproof}
Suppose otherwise. Let $I=( a_i)_{i\in \omega}$ be an infinite Morley sequence of $\tp(a'/b)$ and let $d_i\in \acl(ba_i)\cap H_{a_i}$. Since $\thr(b/I)$ is finite, for some $n$ we have $d_{n+1}\in \acl(d_1,\dots,d_n,I)$. But then $\{H_{a_i}\}_{i\in \omega}$ are not orthogonal, a contradiction to Claim \ref{I}.
\end{claimproof}

We will prove, by a back and forth argument, that for any finite tuple $\bar b$
\[
\qfL(\acl(\bar b))\cup \qfr (\acl(\bar b))\vdash \tp^{\mathcal L\cup \{R\}}(\acl(\bar b)),
\]
where $\acl(\bar b)$ is understood in the original $\mathcal L$-structure $M$ (we will use this notation from now on).

Assume that for finite tuples $\bar b$ and $\bar d$ we have an $\mathcal L\cup \{R\}$-embedding $f$ from $\acl(\bar b)$ to $\acl(\bar d)$. Notice that
by quantifier elimination in $\mathcal L$, $\tp^\mathcal L (\acl(\bar b))=\tp^\mathcal L(\acl(\bar d))$. Let $\beta\in M$ be any element, and we need to find some $\delta$ such that
\[
\qfL(\acl(\bar b \beta))\cup \qfr (\acl(\bar b \beta))=\qfL(\acl(\bar d \delta))\cup \qfr (\acl(\bar d \delta)).
\]

Let $X=p(M)$ and let
\[\bar \alpha=(\alpha_1, \dots, \alpha_k)=\left(\acl\left(\bar b\beta\right) \setminus \acl\left(\bar b\right)\right)\cap X.
\]

By quantifier elimination in $\mathcal L$, there is a sequence $\bar \gamma =(\gamma_1,\dots, \gamma_k)$ such that
\[
\qfL\left(\tp\left(\bar b \bar \alpha\right)\right)=\qfL\left(\tp\left(\bar d \bar \gamma\right)\right).
\]
Extend $f$ so that it maps $\acl(\bar b\bar \alpha)$ to $\acl(\bar d\bar \gamma)$ and is an $\mathcal L$-embedding. By Claim \ref{I} we have $H_{\alpha_i}\cap \acl(\bar\alpha)=\emptyset$ and by Claim \ref{III}, $H_{\alpha_i}\cap \acl(\alpha_i\bar b )=\emptyset$. By Theorem \ref{modularity}, $H_{\alpha_i}\cap \acl(\bar \alpha \bar b )=\emptyset$. It follows that
no 3-tuple in $\acl(\bar b \bar \alpha)^3\setminus \acl(\bar b)^3$ satisfies $R$, so $f$ is an $\mathcal L\cup \{R\}$-embedding.

\medskip

Now, let $\bar t_i=(t_i^1, \dots, t_i^{n_i})$ enumerate $\acl(\bar b \beta)\cap H_{\alpha_i}$. By orthogonality and indiscernibility of the sets $H_{\alpha_i}$, the sequences $( \bar t_i)_{i\leq k}$ are mutually totally $\mathcal L$-indiscernible. For each $i\leq k$ let $\bar s_i$ be such that $\bar s_i$ satisfies the same $R(x,y,\gamma_i)$-order type as $\bar t_i$ (which can always be found by density). Then the sequences $( \bar s_i)_{i\leq k}$ are also mutually totally $\mathcal L$-indiscernible, so that
\[
\qfL\left(( \bar t_i),\acl\left(\bar b, \bar \alpha\right)\right)=\qfL\left(( \bar s_i),\acl\left(\bar d, \bar \gamma\right)\right).
\]
Again, by $\mathcal L$-quantifier elimination, there is some $\delta$ such that
\[\qfL\left(\beta,( \bar t_i),\acl\left(\bar b, \bar \alpha\right)\right)=\qfL\left(\delta,( \bar s_i),\acl\left(\bar d, \bar \gamma\right)\right).\]

It follows that
\[\qfL\left( \acl\left(\bar b \beta\right)\right)=\qfL\left( \acl\left(\bar d \delta\right)\right)\]
and since, $R(\tau_1, \tau_2, a')$ holds for $\tau_1,\tau_2, a'\in \acl(\bar b \beta)\setminus \acl(\bar b)$ only if $a'=\alpha_i$ and $\tau_1, \tau_2=t_i^{j_1}, t_i^{j_2}$ for some $i, j_1, j_2$, we have by construction and hypothesis that
\[\qfr\left( \acl\left(\bar b \beta\right)\right)=\qfr\left( \acl\left(\bar d \delta\right)\right),\]
as required.

Having established quantifier elimination, dependence follows at once since every $\mathcal L$-formula is dependent and $R$ is dependent (with any partition of the variables).
\end{proof}

\bigskip

Notice that in the proof above we actually prove that under the conditions stated there, for any tuple $\bar b$ and any element $d$, if $\bar \alpha$ is the tuple of elements in $\acl(\bar b \beta)\setminus \acl(\bar b)$ satisfying $\tp(a/\emptyset)$ and $\bar t_i :=\acl(\beta\alpha_i)\cap H_{\alpha_i}$ then
$\tp^{\mathcal L\cup {R}}(\beta/\bar b)$ is implied by
\[
\qfL(\beta, \bar t_i, \bar \alpha_i/\bar b)\cup \bigcup_i \qfr(\bar t_i, \alpha_i).
\]

\begin{prop}\label{R algebraic}
Let $M, \mathcal L, R, p(x), a$ and $H_a$ be as in Theorem \ref{adding R}. Then, any $\mathcal L\cup \{R\}$-imaginary is interalgebraic with an $\mathcal L$-imaginary.
\end{prop}

\begin{proof}
As in Theorem \ref{adding R}, we use $\acl$ to mean $\acl^\mathcal L$. For any $a'\models p(x)$ and $s,t\in H_{a'}$ we write $s<_{a'}^R t$ for $\models R(s,t,a')$.

Let $E$ be an $\mathcal L\cup \{R\}$-definable equivalence relation and let $b E d$. We will show that given any $b'\models \tp^\mathcal L (b/d)$ such that $\qfr(\acl(b)/\emptyset)=\qfr(\acl(b')/\emptyset)$ we have $b'E d$. Since there are only finitely many possibilities for  $\qfr(\acl(b)/\emptyset)$, the result will follow.

Let $\bar \alpha=(\alpha_1, \dots, \alpha_n)$ be the tuple of elements in $\acl(bd)$ realizing $p(x)$, and let $\bar {s_i}=(\acl(bd\alpha_i)\setminus \acl(d\alpha_i))\cap H_{\alpha_i}$. By Theorem \ref{modularity} (stable triviality), $\bar {s_i}=(\acl(b\alpha_i)\setminus \acl(d\alpha_i))\cap H_{\alpha_i}$.

As in Theorem \ref{adding R}, $\tp^{\mathcal L}(\bar \alpha/d)\vdash \tp(\bar \alpha/d)$, so by homogeneity, for any $b'\equiv_d^{\mathcal L} b$  there is an automorphism fixing $d$ and sending $b'$ to some element $b''$ such that $\bar \alpha \in \acl(b''d)$, so we may assume without loss of generality that $\bar \alpha \in \acl (b'd)$.

Also,
\[
\tp^{\mathcal L}(b/d\alpha_1,\dots,\alpha_n, \bar{s_1}, \dots,\bar{s_n})\vdash \tp(b/d\alpha_1,\dots,\alpha_n, \bar{s_1}, \dots,\bar{s_n}),
\]
which means that it is enough to show that for any $b'\equiv_d^{\mathcal L} b$ such that $\alpha_i'=\alpha_i$ and
such that for each $i$, $\bar{s_i}'=\acl(b'\alpha_i)\cap H_{\alpha_i}$ has the same $<_{\alpha_i}^R$-order type as $\bar{s_i}$, then $b' E d$.

\begin{claim}
Let $b'\equiv_d^{\mathcal L} b$ such that $\alpha_i'=\alpha_i$ and
such that for each $i$, $\bar{s_i}'=\bar{s_i}$ except for one element in one of the tuples.
Then $b'E d$.
\end{claim}

\begin{claimproof}
We may assume that $\bar{s_i}'=\bar{s_i}$ for $i\geq 2$, and that $\bar{s_1}'=(\tau_1', \tau_2, \dots, \tau_k)$, $\bar{s_1}= (\tau_1, \tau_2, \dots, \tau_k)$ and $\tau_1'<\tau_1$. All other cases are similar.

Let $\bar \gamma=\{ x\in \acl(d\alpha_1)\mid  x<_{\alpha_1}^R \tau_1\}$ and $\bar t=\acl(d\alpha_1)\setminus \bar \gamma$.

Let $\sigma$ be an $<_{\alpha_1}^R$-automorphism of $H_{\alpha_1}$ which fixes $\tau_1,\tau_2,\dots, \tau_k,\bar t$ and sends $\bar \gamma$ to elements $<_{\alpha_1}^R$-smaller than $\tau_1'$.

For $i\geq 1$, let $\bar t_i=\acl(d\alpha_i)\cap H_{\alpha_i}$

It follows that
\[\qfr(\bar t_1, \bar t_2 \dots, \bar t_i,\dots  /\bar \alpha b)=\qfr(\sigma(\bar\gamma)\widehat{~}\bar t, \bar t_2 \dots, \bar t_i,\dots /\bar \alpha b)\] by construction and orthogonality of the $H_{\alpha_i}$, and
\[\qfL(\bar t_1, \bar t_2 \dots, \bar t_i,\dots/\bar \alpha bb' )=\qfL(\sigma(\bar\gamma)\widehat{~}\bar t, \bar t_2 \dots, \bar t_i,\dots/\bar \alpha bb')\] because $H_{\alpha_i}$ are mutually indiscernible sets in $\mathcal L$.

As in Theorem \ref{adding R} this implies that there is some $d'$ with $\acl(d'\alpha_1)\cap H_{\alpha_1}=\sigma(\gamma)\widehat{~} \bar t$, $\acl(d'\alpha_i)=t_i$ for all $i\geq 2$ and such that $\qfL(d'/\bar \alpha bb' )=\qfL(d/\bar \alpha bb')$. This implies that
$\tp(d'/\bar \alpha b)=\tp(d/\bar \alpha b)$ so that $d'Eb$.

But now $\tau_1$ and $\tau_1'$ have the same $<_{\alpha_1}^R$-type with respect to $\sigma(\gamma)\widehat{~}t=\acl(d'\alpha_1)\cap H_{\alpha_1}$, and these are the only elements in which $\acl(b')\cap \bigcup_i H_{\alpha_i}$ and $\acl(b)\cap \bigcup_i H_{\alpha_i}$ differ, so that
\[
\qfr(\acl(b)/\acl(d'))=\qfr(\acl(b')/\acl(d')),
\]
Since $\qfL(\acl(b)/\acl(d'))=\qfL(\acl(b')/\acl(d'))$ by construction of $d'$, we have by homogeneity that $\tp(b/d')=\tp(b'/d')$, so $b'Ed'$. By symmetry and transitivity, $b'Ed$, as required.
\end{claimproof}

Given any sets of disjoint ordered sets $(H_1, <_1), (H_2, <_2), \dots, (H_n, <_n)$ and any two sets $S_1$ and $S_2$ of $<_i$-ordered tuples with the same size and order type, we can choose a sequence of sets of tuples which begins with $S_1$, ends with $S_2$ and such that $S^i$ and $S^{i+1}$ differs by just one element. The proposition now follows from the claim.
\end{proof}

\begin{prop}\label{R-th-rank}
Let $M$ and $M^R$ be as in the statement of Theorem \ref{adding R}. Then \th-independence between sets and tuples in $M$ coincides in $M$ and $M^R$.
\end{prop}

\begin{proof}
Let $b, d$ be two tuples in $M$ and let $A\subseteq M$.

By Theorem \ref{modularity}, $a\thind_A b$ in $M$ if any $\mathcal L$-imaginary in $\acl(a)\cap \acl(b)$ is algebraic over $A$. By Proposition \ref{R algebraic}, this happens if and only if any $\mathcal L\cup \{R\}$-imaginary in $\acl(a)\cap \acl(b)$ is algebraic (in $M^R$) over $A$, which again by  Theorem \ref{modularity} is equivalent to having $a\thind_A b$ in $M^{R}$.
\end{proof}

\begin{theorem}\label{th:distal exp}
Let $M$ be a dependent structure of finite $\thr$-rank admitting quantifier elimination in a finite relational language $\mathcal L$. Then there is a
distal, dependent, finitely homogeneous structure $N$ such that $M$ is a reduct of $N$.
\end{theorem}

\begin{proof}
Let $\{p_i(x,\bar z)\}_{i<N}$ be an enumeration of all complete types such that $p_i(c, \bar a)$ implies that $\tp(c/\bar a)$ is strictly minimal and such that if $p_i(c_i, \bar a_i)$, $p_j(c_j, \bar a_j)$ and $\thr(\tp(\bar a_i/\emptyset))<\thr(\tp(\bar a_j/\emptyset))$ all hold, then $i<j$.

Let $M_0=M$, $\mathcal L_0=\mathcal L$. Suppose that we have defined an expansion $\mathcal L_n$ of $\mathcal L$ and an $\mathcal L_n$-structure on $M$ such that for all $i< n$
and $c', \bar a'\in M$ satisfying $p_i$, the type $\tp^{\mathcal L_n}(c'/\bar a')$ is unstable.

Let $c,\bar a\models p_n(x,z)$. Since $\tp^{\mathcal L}(c/\bar a)$ is strictly minimal by construction, and \th-independence between elements in $M$ coincides in $M_0$ and $M_{n+1}$ (Proposition \ref{R-th-rank}), $\tp^{\mathcal L_n}(c/\bar a)$ is a \th-rank one set. If it is unstable (in $\mathcal L_n$) we let $\mathcal L_{n}=\mathcal L_{n+1}$. If it is stable, then there are tuples $\gamma$ and $\alpha$ $\mathcal L_n$-interalgebraic with $c$ and $\bar a$ such that $\tp^{\mathcal L_n}(\gamma/\alpha)$ is strictly minimal.

Now, by hypothesis on the $\mathcal L_n$-structure $M$ and on the enumeration of the $p_i$, all types over sets of \th-rank less than
\[
\thr^{\mathcal L}(\bar a/\emptyset)=\thr^{\mathcal L}(\alpha/\emptyset)=\thr^{\mathcal L_{n}}(\alpha/\emptyset)
\] are unstable, so in particular not strictly minimal. So the definable subset $H_\alpha:=\tp(\gamma/\alpha)$ in the $\mathcal L_n$-structure $M$ satisfies the conditions of Proposition \ref{adding R}. If we let $R(x,y,z)$ and $M^R$ be as in the proposition, then in the $\mathcal L_{n+1}:=\mathcal L\cup \{R\}$-structure $M^R$ the type $\tp(\gamma/\alpha)$ is unstable. By interalgebraicity (and the fact that algebraic closure is preserved) $\tp(c/\bar a)$ is unstable, as required.

Consider now the $\mathcal L'=\bigcup_n \mathcal L_n$-structure on $M$. We prove that every type is distal. Let $a$ be any tuple in $M$. By Proposition \ref{homogeneous coordinatization} there is an $M^{eq}$-tuple $(\alpha_0, \dots, \alpha_n)$ which is $\mathcal L$-interalgebraic with $a$ and such that $\tp^{\mathcal L}(\alpha_{i+1}/\alpha_i)$  is either
finite,  unstable of \th-rank one, or stable and strictly minimal. By construction $\tp^{\mathcal L'}(\alpha_{i+1}/\alpha_i)$ is either finite, or unstable. Since $a$ and $(\alpha_0, \dots, \alpha_n)$ are interalgebraic in $\mathcal L'$, it follows from facts \ref{fact:distal trans} and \ref{fact:distal acl} that $\tp(a/\emptyset)$ is distal, as required.
\end{proof}

\begin{cor}
There are, up to inter-definability, at most countably many structures $M$ which are homogeneous in a finite relational language, dependent and rosy.
\end{cor}
\begin{proof}
By Proposition \ref{prop:finite rank homogeneous}, any such structure has finite $\thr$-rank. By Theorem \ref{th:distal exp}, they are all reducts of distal finitely homogeneous structures. By \cite[Theorem 8.3]{Si}, all such structures are finitely axiomatizable. It follows that there are only countably many. Finally, each has countably many reducts.
\end{proof}

\bibliography{binary}

\end{document}